\documentclass[12pt]{article}
\usepackage{amsmath,amssymb}
\usepackage{amsthm}
\usepackage{enumitem}
\usepackage{amsmath}

\def\widebreve#1{\mathop{\vbox{\m@th\ialign{##\crcr\noalign{\kern3\p@}%
      \brevefill\crcr\noalign{\kern3\p@\nointerlineskip}%
      $\hfil\displaystyle{#1}\hfil$\crcr}}}\limits}
\def\brevefill{$\m@th \setbox\z@\hbox{$\braceld$}%
  \bracelu\leaders\vrule \@height\ht\z@ \@depth\z@\hfill\braceru$}
\makeatletter
\newtheorem{theorem}{Theorem}[section]
\newtheorem*{theorem*}{Theorem}
\newtheorem{proposition}[theorem]{Proposition}
\newtheorem{corollary}[theorem]{Corollary}
\newtheorem{lemma}[theorem]{Lemma}

\newtheorem{hypothesis}[theorem]{Hypothesis}
\theoremstyle{definition}

\newtheorem*{remark*}{Remark}

\newtheorem*{observation*}{}
\makeatletter
\@addtoreset{equation}{section}
\makeatother

\newcommand{\noin}{\noindent}

\providecommand{\AMS}{$\mathcal{A}$\kern-.1667em%
\lower.25em\hbox{$\mathcal{M}$}\kern-.125em$\mathcal{S}$}

\title{Homomorphism extension problem for subdirect products of finite groups} 

\author{\large İsmail Alperen Öğüt \footnote{e-mail iaogut@ankara.edu.tr}\\
\mbox{}\\
Department of Electronics and Automation \\
Ankara University \\
Vocational School of 1st Organized Industrial Zone, Ankara, Turkey \\
\mbox{}}
\begin{document}
\title{Homomorphism extension problem for subdirect products of finite groups}

\maketitle

\small
\begin{abstract}
\noin Motivated by the simplification of decomposition formulas for fibred bisets, we study the homomorphism extension problem for subdirect products of finite groups when the codomain is an abelian group satisfying certain hypothesis. We prove that every homomorphism of subdirect products whose Goursat quotients have trivial Schur multipliers is extensible. We also examine the case where subdirect products contain a twisted diagonal subgroup and investigate the functoriality of the extensibility property.
\smallskip\\
\noin 2010 {\it Mathematics Subject Classification:}
Primary 20C15, Secondary 16B50.

\smallskip
\noin {\it Keywords:} Goursat quotient, Schur multiplier, subdirect product, homomorphism extension, fibred biset. 
\end{abstract}
\section{Introduction}
Functorial approaches in the finite group representation theory often involve categories where the morphisms are associated with the subgroups of direct products of finite groups. Some important examples of such categories are the biset \cite{BOU96} and the fibred biset category \cite{BC18} and the category of the diagonal $p$-permutation modules \cite{BY22}. Besides being interesting objects on their own, such categories provide tools for studying many representation theoretic notions by realizing them as their functors to the module category. 

One common approach for studying these categories is decomposing their morphisms as products of simpler ones. In the case of the biset category, the decomposition is into five basic morphisms: induction, inflation, isogation, deflation and restriction. The fibred biset category, on the other hand, is more general. It is obtained by equipping the morphisms of the biset category with one dimensional characters. Although this added structure enriches the category, it also brings certain inconveniences. In particular, the general decomposition formula one gets is not as neat as the one for the non-fibred case. As shown in \cite[Lemma 5.1]{BY21}, obtaining a neat decomposition via the factorization of fibred bisets over the Goursat quotient is equivalent to the property of character extension. That is, for the associated pair $(U,\varphi)$ where $U$ is a subgroup of a direct product of finite groups $G\times H$, the character $\varphi$ must extend to $G\times H$.

By considering particular subcategories of the fibred biset category, it is possible to obtain an environment where the decomposition is simplified. For the full subcategory on the abelian groups, it is possible to realize this simplification \cite[Theorem $1$]{CY19}, whereas by restricting to cyclic groups of prime order, an even simpler decomposition can be obtained \cite[Theorem $3.4$]{AC20}, both of which are cases where the extension is possible for every associated pair. In \cite[Section $5$]{BY21}, this extension problem is studied in a more general setting where the only hypothesis on the codomain is that the group of homomorphisms from any finite group $G$ into it is finite. However, instead of exploring the classes of finite groups where the extension is possible for every associated pair, the authors focus on conditions that ensure extensibility for a particular pair, characterizing it in terms of the nonvanishing of certain equations involving primitive idempotents.

The above examples motivate us to explore the necessary and sufficient conditions for every homomorphism of $U$ to extend to $G\times H$. Our codomains will be abelian groups satisfying the following standard hypothesis.
\begin{hypothesis}\label{hyp}\cite[Hypothesis 10.1]{BC18}
There exists a unique set of primes $\pi$ such that for every $n\in \mathbb{N}$, the $n$-torsion part of $A$ is cyclic of order $n_\pi$. 
\end{hypothesis}

Let us clarify why we focus only on subdirect products, (subgroups $U\leq G\times H$ for which the projections onto the first and the second coordinates are $G$ and $H$, respectively) rather than every subgroup $U\leq G\times H$. For subgroups that are not subdirect products, once the extension to the direct product of the projections is achieved, it remains checking the one dimensional extension (between each projection subgroup and $G$ and $H$, respectively) which is a special case of the main problem. Interestingly, a result that solves the character extension problem for subdirect products can be applied to \cite[Lemma 5.1]{BY21} to obtain an explicit description for idempotent composition.

To make our study of character extension problem precise, we introduce the following notions of extensibility. A subgroup $U\leq G\times H$ is called \textbf{$A$-extensible} if every $\varphi\in \mathrm{Hom}(U,A)$ extends to $G\times H$. In the case where $U$ is $A$-extensible for every $A$ satisfying Hypothesis \ref{hyp} we say $U$ is \textbf{extensible}. When  $U$ is $A$-extensible for every abelian group $A$ satisfying the hypothesis for a set $\pi$ consisting of a single prime $p$, we say $U$ is \textbf{$p$-extensible}.

Our main tool will be a necessary and sufficient condition for extensibility in terms of the commutator and the kernel subgroups (Theorem \ref{thm3.2}). Following its establishment, we will then study the implications of it on the Goursat quotient (denoted by $q(U)$, see the beginning of the next section for a definition) and prove the following main result of the paper. The definition of the coinflation map can be found in Section 4. 
\enlargethispage{2\baselineskip}
\begin{theorem}\label{mainthm1}
Let $p$ be a prime, $G,H$ finite groups, $U\leq G\times H$ a subdirect product. Then $U$ is $p$-extensible if one of the following conditions hold:
\begin{enumerate}[label=(\roman*)]
\item The $p$-primary part of the Schur multiplier of $q(U)$ is trivial.
\item The $p$-primary component of the coinflation map is surjective.
\end{enumerate}
\end{theorem}

Obtaining a necessary and sufficient condition for the extension of the homomorphisms alone would not be very useful for functorial applications as we would like these conditions to be preserved under the composition of the category in question. There are several cases where the $p$-part of the Schur multiplier vanishes (we refer the reader to \cite{K87} for a comprehensive source). The case of Goursat quotients having cyclic Sylow $p$-subgroups for each prime divisor of the quotient is an example of a property that is \emph{functorial} (it survives the categorical composition), which we include as Corollary $\ref{cor4.1}$. We note that it is highly likely that there are more interesting or general classes that are functorial. However, discovery of these is expected to require a delicate combination of existing results for the vanishing of the Schur multiplier.

In the last section, we study another class of subgroups for which our methods are suitable, in particular, the class  defined by the condition of containment of a twisted diagonal subgroup. The main result (Theorem \ref{thm5.6}) of the section provides a convenient condition for determining whether the extensibility is preserved under the categorical composition of particular elements from the class. 
 
\section{Preliminaries}
Let $G$ and $H$ be finite groups. Let us recall the parameters that determine  $U\leq G\times H$.
\begin{lemma}[Goursat]
  There exists a one to one correspondence between the subgroups $U\leq G\times H$ and the quintuples $\Delta(p_1(U), k_1(U),\varphi, k_2(U), p_2(U))$ where
  \begin{align*}
    p_1(U) &= \lbrace g\in G| \exists h\in H, (g,h) \in U \rbrace\\
    k_1(U) &= \lbrace g\in G| (g,1) \in U \rbrace\\
    p_2(U) &= \lbrace h\in H| \exists g\in G, (g,h) \in U\rbrace \\
    k_2(U) &= \lbrace h\in H| (1,h) \in U \rbrace \\
  \end{align*}
  and $\varphi : p_1(U)/k_1(U) \rightarrow p_2(U)/k_2(U)$ is the isomorphism given by $\varphi(g,h) = (gk_1(U),hk_2(U))$ for every $(g,h)\in U$.
\end{lemma}
The isomorphism class $p_1(U)/k_1(U)$ will be denoted by $q(U)$. We will refer to it as the Goursat quotient of $U$. 

Given finite groups $F,G,H$, and subgroups $U\leq F\times G, V\leq G\times H$ the \textbf{$\ast$-product} of $U$ and $V$ is the subgroup $U\ast V \leq F\times H$ given by
$$
\lbrace (u,v)\in p_1(U)\times p_2(V)|\exists x\in p_2(U)\cap p_1(V), (u,x)\in U, (x,v)\in V\rbrace.
$$
For the fibred biset category, functoriality of a property (assuming it does not pose any restriction on the characters) is equivalent with its preservation under the $\ast$-product.

Notice that $k_1(U)\leq k_1(U\ast V)$ and $p_1(U)\geq p_1(U\ast V)$. This yields the following.
\begin{proposition}{\label{prop2.2}}
    Let $U, V$ and $W$ be finite groups such that $W = U\ast V$. Then 
    $$
    \left[q(U)\right] \geq \left[q(W)\right]\leq \left[q(V)\right].
    $$
\end{proposition}

For a finite group $G$, let $G'$ denote its commutator subgroup and $\overline{G}\cong G/G'$ its abelianization. We include the following for ease of reference.
\begin{proposition}\label{prop22}
  Let $G$ be a finite group, $N\trianglelefteq G$.Then 
  $$\bigg(\dfrac{G}{N}\bigg)'\cong \dfrac{G'}{G'\cap N}.$$
\end{proposition}
\begin{proof}
	The projection homomorphism
	$$
	\pi: G'\rightarrow (G/N)'
	$$
	is clearly surjective and its kernel is exactly $G'\cap N$, which yields the result.
\end{proof}
 We are going to be using the following facts regarding the structure of $U$ and its commutator subgroup.
\begin{lemma}\label{lem23}
  Let $G$ and $H$ be finite groups, $U\leq G\times H$ and $i,j\in \lbrace 1,2\rbrace, i\not= j$. Then the following holds:
  \begin{enumerate} [label=(\roman*)]
    \item $|U| = |p_i(U)|\cdot |k_j(U)|$
    \item $p_i(U') = (p_i(U))'$
    \item $[k_i(U),p_i(U)]\leq k_i(U')\leq (p_i(U))'\cap k_i(U).$
  \end{enumerate}
\end{lemma}
\begin{proof}
	We will provide the proofs for the case $i=1$, the other case is follows from symmetry.
	
	The first assertion follows from the fact that given $(g,h)\in U$, for each $k\in k_2(U)$, one has $(g,hk)\in U$ and if $(g,h')\in U$, then $(1,h^{-1}h')\in U, h^{-1}h'\in k_2(U)$, so there is exactly $|k_2(U)|$ elements in $U$ whose first coordinate is $g$. 
	
	The second assertion is a consequence of the fact that for any homomorphism $\varphi$ from $U$, we have $$
	\varphi(U') = (\varphi(U))'.
	$$
	
	For the first inclusion of the last statement, let $k\in k_1(U), (g,h)\in U$. Then $$[(k,1),(g,h)] = [(k,g),1]\in k_1(U').$$ The second inclusion follows from the fact that $k_1(U)\geq k_1(U')\leq p_1(U')=(p_1(U))'$.
	
\end{proof}
\section{The Necessary and Sufficient Condition for Extensibility}
In this section we will obtain the necessary and sufficient conditions for every homomorphism of $U$ to extend to $G\times H$. Recall that a \textbf{subdirect product} of $G$ and $H$ is a subgroup $U\leq G\times H$ such that $p_1(U)= G, p_2(U) = H$.
\begin{lemma}
  Let $G$ and $H$ be finite groups, $U\leq G\times H$ a subdirect product, $A$ any (non necessarily finite or abelian) group. The kernel of the natural projection 
  $$
  \rho : \mathrm{Hom}(G\times H, A) \rightarrow \mathrm{Hom}(U, A)
  $$
  is given by
  $$
  \mathrm{ker}(\rho) \cong \mathrm{Hom}(q(U), A).
  $$
\end{lemma}\label{prop24}
\begin{proof}
Let $(\varphi_1 \times \varphi_2)\in \mathrm{Hom}(G\times H,A)$. Suppose for every $(u_1,u_2)\in U$ we have 
$$(\varphi_1\times\varphi_2) (u_1,u_2) = \varphi_1(u_1)\varphi_2(u_2) = 1.$$ Then for every $k_1\in k_1(U), k_2\in k_2(U)$  one has  
$$
\varphi_1(k_1) = 1, \varphi_2(k_2) = 1.
$$
Hence, homomorphisms of $p_1(U)$ and $p_2(U)$ whose product fall into the kernel must factor through $k_1(U)$ and $k_2(U)$, respectively. Moreover, given $\varphi_1\in\mathrm{Hom}(G,A)$ which factors through $k_1(U)$, if there exists $\varphi\in\mathrm{Hom}(H,A)$ such that $\varphi_1\times\varphi_2\in\mathrm{ker}(\rho)$, then for every $(g,h)\in U$, we have $$\varphi_2(h) = \varphi_1(g)^{-1}.$$ Since $U$ is a subdirect product, $\varphi_2$ is completely determined, so if it exists, it must be unique. To see that it exists, notice that in the above equation, every choice of $g$ yields the same value for $\varphi_2(h)$ because $\varphi_1$ factors through $k_1(U)$, hence it defines a function from $H$ to $A$. Since $A$ is an abelian group, this function is actually a homomorphism. Hence, for every such $\varphi_1$, there exists unique $\varphi_2$ factoring through $k_2(U)$ for which $\varphi_1\times\varphi_2\in \mathrm{ker}(\rho)$.
\end{proof}
When $A$ is an abelian group satisfying Hypothesis \ref{hyp}, the above proposition allows us to calculate the order of the image of $\rho$ and one can check the surjectivity of $\rho$ by comparing it with the order of $\mathrm{Hom}(U,A)$. Under Hypothesis \ref{hyp}, the first part of \cite[Proposition $10.4$] {BC18} tells that for any finite abelian group $B$, one has $|\mathrm{Hom}(B,A)|=|B|_\pi$ which yields the following result.

\begin{theorem}\label{thm3.2}
  Let $G$ and $H$ be finite groups, $U\leq G\times H$ a subdirect product, $A$ an abelian group satisfying Hypothesis \ref{hyp} where the set $\pi$ of primes includes every prime divisor of $G$ and $H$. Then the natural projection map $\rho$ is surjective if and only if $G'\cap k_1(U) = k_1(U')$ or equivalently $H'\cap k_2(U) = k_2(U')$.
\end{theorem}
\begin{proof}
By the above lemma and the hypothesis, the surjectivity is equivalent with
 $$
\dfrac{|\overline{G\times H}|}{|\overline{q(U)}|} = | \overline{U}|.
  $$
  Proposition \ref{prop22} yields
  $$
|\overline{q(U)}|=|\overline{G/k_1(U)}|=\dfrac{|G|\cdot|G'\cap k_1(U)|}{|k_1(U)|\cdot|G'|}
  $$
  which yields
  $$
  \dfrac{|\overline{G\times H}|}{|\overline{q(U)}|} = \dfrac{|G|\cdot |H|\cdot |k_1(U)|\cdot |G'|}{|G'|\cdot |H'|\cdot |G|\cdot |G'\cap k_1(U)|} =\dfrac{|H|\cdot |k_1(U)|}{|H'|\cdot |G'\cap k_1(U)|}.
  $$
  Using the first and the second parts of Lemma \ref{lem23} to calculate $|\overline{U}|$  gives
  $$
  |\overline{U}| = \dfrac{|U|}{|U'|}=\dfrac{|p_2(U)|\cdot |k_1(U)|}{|p_2(U')|\cdot |k_1(U')|}=\dfrac{|H|\cdot |k_1(U)|}{|H'|\cdot |k_1(U')|}.
  $$
  So the surjectivity is equivalent with the equality of the orders of the groups $G'\cap k_1(U)$ and $k_1(U')$. The last part of Lemma \ref{lem23} yields the equality of the groups. The statement including $H'$ follows from symmetry.
  
\end{proof}
Below is a version of the above theorem in the case where $\pi$ does not include every prime divisor of $G$ and $H$. It can be obtained by replacing the orders with their $\pi$-parts.
\begin{corollary}{\label{cor2.7}}
Let $U\leq G\times H$ be a subdirect product of $G$ and $H$. Then $U$ is $A$-extensible if and only if for every $p\in\pi$, $k_1(U')$ contains a Sylow $p$-subgroup of $G'\cap k_1(U)$ or equivalently, $k_2(U')$ contains a Sylow $p$-subgroup of $H'\cap k_2(U)$.
\end{corollary}
\section{Extensibility via criteria on $q(U)$}
Having established a necessary and sufficient condition, we are ready to explore the classes of extensible subgroups. In general, it is much harder to describe $k_i(U')$ than $[k_i(U), p_i(U)]$ and $(p_i(U))'\cap k_i(U)$. In the following calculations, we will ignore the complicated term and simply check whether $[k_i(U), p_i(U)]=(p_i(U))'\cap k_i(U))$ holds, which by the inequality in the last part of Lemma \ref{lem23} gives the extensibility.

Given a group extension
$$
1\rightarrow K\rightarrow G\rightarrow G/K \rightarrow 1
$$
there is a five term exact sequence of homology groups
$$
H_2(G) \xrightarrow[]{\alpha} H_2(G/K)\xrightarrow[]{\beta} H_1(K)_{G/K}\xrightarrow[]{\gamma} H_1(G)\rightarrow H_1(G/K)\rightarrow 0
$$
obtained from the Hochschild-Serre spectral sequence \cite{HS53}. Recall that 
$$H_1(K)_{G/K} = \dfrac{K/K'}{<x\cdot {kK'} - kK' : x\in G/K, k\in K>}$$ and $x = gK$ acts on $kK'$ via conjugation by $g$.
The map $\alpha$ is called \textbf{coinflation}.
For a finite group $G$, let $M(G)$ denote its Schur multiplier, i. e., the second homology group $H_2(G)$.

We are now ready to prove Theorem \ref{mainthm1}.
\begin{proof}[Proof of Theorem \ref{mainthm1}]
	Let $U = \Delta(G,K,\varphi, L, H)$.
	Assuming the above notation, we have, $x\cdot {kK'} = gkg^{-1}K'$, hence if $x\cdot {k K'} = kK'$ then $gkg^{-1}k^{-1} \in K'$. Thus we get
	$$
	H_1(K)_{G/K} = \dfrac{K}{[K,G]}.
	$$
	Then 
	$$\mathrm{ker}(\gamma) = \lbrace k[K,G]: k\in K\cap G'\rbrace = \dfrac{K\cap G'}{[K,G]} = \mathrm{im}(\beta).$$
	The proof now follows from Corollary \ref{cor2.7}, the third part of Lemma \ref{lem23} and the fact that $$\dfrac{M(G/K)}{\mathrm{im}(\alpha)}\cong\mathrm{im}(\beta) = \dfrac{K\cap G'}{[K,G]}.$$
\end{proof}

The result below is a direct consequence of \cite[Corollary $2.1.3$]{K87}. We can immediately see that Proposition \ref{prop2.2} ensures that the below sufficient condition is preserved under the $\ast$-product, thus a functorial property. 
\begin{corollary}\label{cor4.1}
Let $G,H$ be finite groups $U\leq G\times H$ a subdirect product. Suppose that for each prime divisor $p$ of $|q(U)|$ the Sylow $p$-subgroups of $q(U)$ are cyclic. Then $U$ is extensible.
\end{corollary} 
\section{The case of containment of twisted diagonals }
In this section we will study the extensibility for subgroups of $G\times G$ which contain a twisted diagonal subgroup, that is, the subgroups $\Delta(G,\varphi)\leq G\times G$ given by
$$
\Delta(G,\varphi) = \lbrace (g,\varphi(g))| g\in G\rbrace
$$ 
for an automorphism $\varphi$ of $G$. Notice that the property of containment of such subgroups is preserved under the $\ast$-product. Another positive fact about them is that kernel subgroups of their commutator subgroups has a nice description which  significantly simplifies the checking of the extensibility. To establish this fact, we need the following lemma.
\begin{lemma}
  Let $G$ be a finite group, $X \trianglelefteq G \trianglerighteq Y$ such that $G = XY$. Then $G' = \langle X', [X,Y], Y'\rangle$.
\end{lemma}
\begin{proof}
  Let $u,v\in G$ and write $u=x_1y_1, v = x_2y_2$. Using the commutator identities $[a,bc] = [a,c][a,b]^c$ and $[ab,c]=[a,c]^b[b,c]$ we get
  \begin{align*}
    [x_1y_1,x_2y_2]&=[x_1,x_2y_2]^{y_1}[y_1,x_2y_2]\\
    &=([x_1,y_2][x_1,x_2]^{y_2})^{y_1}[y_1,y_2][y_1,x_2]^{y_2}\\
    &=[x_1,y_2]^{y_1}[x_1,x_2]^{y_2y_1}[y_1,y_2][y_1,x_2]^{y_2}.
  \end{align*}
The normality of $X$ and $Y$ implies that $[X,Y]^G=[X,Y]$ hence every term of the product for the commutator $[u,v]$ is from the claimed generating set.
  
\end{proof} 
\begin{proposition}\label{thm5.2}
  Let $\Delta(G,\varphi)\leq U\leq G\times G$ for some $\varphi \in \mathrm{Aut}(G)$. Then for $i\in\lbrace 1,2\rbrace$ one has
  $$k_i(U') = [k_i(U),G].$$
\end{proposition}
\begin{proof}
  From the last part of Lemma \ref{lem23}, we obtain $[k_i(U),G]\leq k_i(U')$. For the reverse inclusion, notice that the assumption on $U$ implies that if $(x,y)\in U$ then $(x,y)=(kg,\varphi(g))$ for some $k\in k_1(U),g\in G$ hence $U=(k_1(U)\times 1)\Delta(G,\varphi)$. Now let $k\in k_1(U')$. Observe that applying the above lemma to $U$ implies either $k\in (k_1(U))'$ or  $k\in [k_1(U)\times 1, \Delta(G, \varphi)]$ hence $k\in [k_1(U),G]$.
\end{proof}
The next result provides a useful way to avoid inextensible subdirect products.
\begin{corollary}
Let $\Delta(G,\varphi)\leq U\leq G\times G$ for some $\varphi \in \mathrm{Aut}(G)$. Also suppose that for some $i\in \lbrace 1,2\rbrace$, we have
\begin{enumerate}[label=(\roman*)]
\item  $k_i(U)\leq Z(G)$,
\item $1 < k_i(U)\cap G'$.
\end{enumerate}
Then $U$ is inextensible.
\end{corollary}
\begin{proof}
	The first condition of the hypothesis yields $[k_i(U), G] = 1$, together with the above proposition we get $k_i(U') = 1$. The second condition and Theorem \ref{thm3.2} yields the result.
\end{proof}
We now examine how the extensibility behaves under the $\ast$-product. The next lemma follows from direct calculation.
\begin{lemma}\label{lem5.4}
Suppose that $\Delta(G,\varphi)\leq U\leq G\times G \geq V\geq \Delta(G,\psi)$ for some $\varphi,\psi \in  \mathrm{Aut}(G)$. Then
\begin{enumerate}[label=\textup{(\roman*)}]
\item $k_2(U) = \varphi(k_1(U))$,
\item $k_1(U\ast V) = k_1(U)\varphi^{-1}(k_1(V))$,
\item $k_2(U\ast V) = k_2(V)\psi(k_2(U)).$
\end{enumerate}
\end{lemma}

\begin{lemma}\label{lem5.5}
Let $G$ be a finite group, $U,V\leq G\times G$ such that $U\geq \Delta(G) \leq V$. Then
$$
\dfrac{k_1((U\ast V)')}{k_1(U)\cap k_1((U\ast V)')}\cong \dfrac{k_1(V')}{k_2(U)\cap k_1(V')}
$$
and
$$
\dfrac{k_2((U\ast V)')}{k_2(V)\cap k_2((U\ast V)')} \cong \dfrac{k_2(U')}{k_1(V)\cap k_2(U')}.
$$
\end{lemma}
\begin{proof}
We will show the first isomorphism, the second can be derived from symmetry.

For ease of notation, set $K = k_1((U\ast V)')$. Let $A = \prod_{i=1}^n [a_i,b_i]\in K$ where $n$ is a positive integer, for each $i$, there exists $x_i,y_i,c_i,d_i\in G$ such that $(a_i,x_i),(b_i,y_i)\in U, (x_i,c_i),(y_i,d_i)\in V$ and $\prod_{i=1}^n[c_i,d_i] = 1$. Let $X = \prod_{i=1}^n[x_i,y_i]$. Notice that $X\in k_1(V')$. Define
$$\varphi : K\rightarrow \dfrac{k_1(V')}{k_2(U)\cap k_1(V')}$$ as 
$$
\varphi(A) = X(k_2(U)\cap k_1(V')).
$$
To establish the well-definedness, Let $ \widetilde{X}$ be such that $(A,X),(A,\widetilde{X})\in U$ and $\widetilde{X}\in k_1(V')$. There exists $k\in k_2(U)$ such that $X = k\widetilde{X}$. Since $X,\widetilde{X}\in k_1(V')$ we get $k\in k_1(V')$. Thus $\varphi$ is well-defined. The fact that $\varphi$ is a homomorphism is clear. To calculate the kernel, observe that if $X\in k_2(U)\cap k_1(V')$, since $(A,X)\in U$, we must have $A\in k_1(U)$. 

To see the surjectivity, let $v\in k_1(V')$ and write $v= \prod_{i=1}^n[z_i,t_i]$ for some $z_i, t_i\in G$. Since $U\geq \Delta(G)$, there exists $e_i,f_i\in G$ such that $(e_i, z_i), (f_i, t_i) \in U$. Letting $A = \prod_{i=1}^n[e_i,f_i]$ we see that $A\in k_1((U\ast V)')$ and $\varphi(A) = v$. The proof is now complete.
\end{proof}
As it can be interpreted from the following theorem, the $\ast$-product need not preserve the extensibility. However we have a useful condition to be able to tell when it does.
\begin{theorem}\label{thm5.6}
Let $G$ be a finite group, $U,V\leq G\times G$ and $U\geq \Delta(G)\leq V$. Suppose that $U$ and $V$ are extensible. Then $U\ast V$ is extensible if and only if 
$$k_i(U\ast V)\cap G' =(k_i(U)\cap G')(k_i(V)\cap G')$$ 
for $i\in \lbrace 1,2\rbrace$.
\end{theorem}
\begin{proof}
We shall give the proof for $i =1$, the other case follows from symmetry. For ease of notation set $k_1(U) = N, k_1(V) = M$. By Lemma \ref{lem5.4} we get $k_1(U\ast V) = NM$. 
Lemma \ref{lem5.5} together with Proposition \ref{thm5.2} yields
$$
|k_1((U\ast V)')| = |[NM, G]| = \dfrac{|N\cap [NM,G]|\cdot |[M,G]|}{|N\cap [M,G]|}.
$$ 
Since $N$ and $M$ are normal subgroups, we have $[NM, G]=[N,G][M,G]$. By the last part of Lemma \ref{lem23}, the extensibility of $U$ and $V$ implies that we can write $[N, G] = (N\cap G')$ and $[M, G] = (M\cap G')$ and get
\begin{align*}
|[NM, G]| &= \dfrac{|N\cap (N\cap G')(M\cap G')|\cdot |M\cap G'|}{|N\cap M\cap G'|}\\
&=\dfrac{|N\cap G'|\cdot |M\cap G'|}{|N\cap M\cap G'|}=|(N\cap G')(M\cap G')|
\end{align*}
which finishes the proof.
\end{proof}
\section*{Acknowledgements}
The author would like to thank the anonymous referee for their careful reading of the manuscript and helpful comments.
{}

\end{document}